\documentclass[11pt, a4paper]{article}

\usepackage[utf8]{inputenc}
\usepackage{amsmath, amsthm, amssymb, amsfonts}
\usepackage{geometry}
\usepackage{bm}
\usepackage{algorithm}
\usepackage{algorithmic}
\usepackage{cases}
\usepackage{graphicx}
\usepackage{subcaption}
\usepackage{hyperref}
\usepackage{cleveref}
\usepackage{color}
\usepackage{authblk}

\newtheorem{theorem}{Theorem}[section]
\newtheorem{lemma}[theorem]{Lemma}

\newtheorem{corollary}[theorem]{Corollary}

\newcommand{\vx}{{\bm x}}

\title{The Continuous Relaxation of Sparse PCA is NP-hard}

\author[1]{Linbin Li}
\author[1]{Yong Xia\thanks{Corresponding author. Email: yxia@buaa.edu.cn}}
\affil[1]{LMIB of the Ministry of Education, School of Mathematical Sciences, Beihang University, Beijing, China}

\begin{document}

\maketitle
\begin{abstract}
Maximizing a symmetric quadratic form under simultaneous $L_1$ norm inequality and $L_2$ norm equality constraints is a standard and widely used continuous relaxation for Sparse Principal Component Analysis (SPCA). This paper settles the computational complexity of this continuous formulation by proving it is NP-hard. Furthermore, the variant with both $L_1$ and $L_2$ norm inequalities is also shown to be NP-hard.
\end{abstract}

\vspace{1em}

\section{Introduction}
\label{sec:intro}

Consider the problem of maximizing a symmetric quadratic form subject to simultaneous $L_1$ norm inequality and $L_2$ norm equality constraints:
\begin{equation}
\label{eq:l1l2_intro}
    (L_1\text{-}L_2\text{-QM})\ \max_{\vx \in \mathbb{R}^n} \vx^T \bm{A} \vx \quad \text{s.t.} \quad \|\vx\|_1 \le 1, \;\; \|\vx\|_2 = R,
\end{equation}
where $\bm{A} \in \mathbb{R}^{n \times n}$ is a symmetric matrix and $R>0$ is a given parameter. The problem ($L_1$-$L_2$-QM) was first proposed by Jolliffe et al. \cite{jolliffe2003modified}. Since then, it has found broad applications in modern data science, particularly in dimensionality reduction and the analysis of high-dimensional gene expression data  \cite{journee2010generalized,witten2009penalized}.

The problem ($L_1$-$L_2$-QM) originates from the fundamental problem of finding a sparse principal direction, traditionally formulated with the combinatorial $L_0$ pseudo-norm:
\begin{equation}
    (\text{SPCA}) \ \max_{\vx \in \mathbb{R}^n} \vx^T \bm{A} \vx \quad \text{s.t.} \quad \|\vx\|_0 \le k, \;\; \|\vx\|_2 = 1,
\end{equation}
where $\|\vx\|_0$ counts the number of non-zero elements in $\vx$ and $k$ is an integer. Due to the combinatorial nature, the problem (SPCA) is well known to be NP-hard \cite{d2004direct}.

To circumvent the combinatorial intractability of (SPCA), its continuous relaxation is widely employed. The relaxation from the $L_0$ to the $L_1$ norm is mathematically motivated by the Cauchy-Schwarz inequality:
\begin{equation}
\label{CS}
    \|\vx\|_1  \le \|\vx\|_2  \cdot \sqrt{\|\vx\|_0} \le \sqrt{k}.
\end{equation}
After suitable normalization, we recover ($L_1$-$L_2$-QM) with $R=1/\sqrt{k}$. 

A natural variant of ($L_1$-$L_2$-QM) relaxes the $L_2$ norm equality to an inequality:
\begin{equation}
\label{problem:inequality}
    (L_1\text{-}L_2\text{-QM-Ineq}) \  \max_{\vx \in \mathbb{R}^n} \vx^T \bm{A} \vx \quad \text{s.t.} \quad \|\vx\|_1 \le 1, \;\; \|\vx\|_2 \le R,
\end{equation}
where the feasible region defined by simultaneous $L_1$ and $L_2$ norm inequalities is widely recognized as the constraint space of the Elastic Net \cite{argyriou2012sparse}.

When $R \ge 1$, the $L_2$ constraint becomes redundant due to  $\|\vx\|_2 \le \|\vx\|_1\le 1$. Problem ($L_1$-$L_2$-QM-Ineq) then reduces to the standard $L_1$-constrained quadratic optimization \cite{bomze2007improved,nesterov1998global,pinar2006semidefinite,luss2011convex}:
\begin{equation}
    (\text{QPL1})\ \max_{\bm{x} \in \mathbb{R}^n} \bm{x}^T \bm{A} \bm{x} \quad \text{s.t.} \quad \|\bm{x}\|_1 \le 1.
\end{equation}
Hsia \cite{hsia2014complexity} proved that (QPL1) is NP-hard. Consequently, the problem ($L_1$-$L_2$-QM-Ineq) is NP-hard for $R \ge 1$. However, its complexity when the $L_2$ constraint is active (i.e., $R < 1$) had remained open.

Despite the wide applicability of these formulations, solving them is computationally challenging. For ($L_1$-$L_2$-QM), the combination of an $L_1$ inequality and an $L_2$ equality restricts the feasible region to disconnected components on the unit sphere, creating severe non-convexity and multiple local optima that trap numerical algorithms \cite{trendafilov2006projected}. For the inequality variant ($L_1$-$L_2$-QM-Ineq), although its feasible region is convex, the objective function can be indefinite, which prevents straightforward convex optimization techniques from guaranteeing global optimality.

To the best of our knowledge, the exact computational complexity of ($L_1$-$L_2$-QM) has remained unaddressed in the literature. Does the introduction of the $L_2$ norm equality (or inequality) constraint fundamentally alter the hardness of the continuous relaxation? In this paper, we fill this gap by proving that both are NP-hard.

\section{Main results}
\label{sec:l1l2_intersection}

We establish the computational complexity of the ($L_1$-$L_2$-QM) and ($L_1$-$L_2$-QM-Ineq) problems. To this end, we first recall a classical result from graph theory and present a structural lemma concerning the allowable parameter regime, which together lay the foundation for our main theorem. Throughout the article, let bold lowercase letters (e.g., $\vx$) denote vectors and italic lowercase letters (e.g., $x_i$) denote their components. Let $\bm{A}$ be a matrix with entries $A_{ij}$.

\begin{lemma}[Motzkin-Straus Theorem \cite{motzkin1965maxima}]
\label{theorem:Motzkin-Straus}
    Consider an undirected graph $G=(V,E)$ with $n$ vertices and adjacency matrix $\bm{A}$, where $\bm{A}_{ij}=1$ if $(i,j)\in E$, and $\bm{A}_{ij}=0$ otherwise. Then
    \begin{equation*}
        \max_{\vx \in \Delta} \vx^T \bm{A} \vx = 1 - \frac{1}{\omega(G)},
    \end{equation*}
    where $\Delta = \{\vx \in \mathbb{R}^n \mid \sum_{i=1}^n x_i = 1, x_i \ge 0~\forall i\}$ is the standard simplex and $\omega(G)$ is the order of the maximal complete graph contained in $G$.
\end{lemma}

To isolate the genuinely hard instances of ($L_1$-$L_2$-QM), we first identify the parameter regimes for which the problem is either infeasible or solvable in polynomial time.

\begin{lemma}
\label{lem:range_R}
    The 
    problem ($L_1$-$L_2$-QM) is either infeasible or solvable in polynomial time  unless the parameter  $R$ satisfies
     $1/\sqrt{n} < R < 1$.
\end{lemma}
\begin{proof}
    For any vector $\vx \in \mathbb{R}^n$, standard norm inequalities dictate that $\|\vx\|_2 \le \|\vx\|_1 \le \sqrt{n}\|\vx\|_2$. Applying the constraints of (L1-L2-QM), we have $R \le \|\vx\|_1 \le \sqrt{n}R$.

    If $R > 1$, then $\|\vx\|_1 \ge R > 1$, which strictly violates the constraint $\|\vx\|_1 \le 1$. Hence the feasible region is empty, and the problem is trivially solved.

    If $R = 1$, the constraints $\|\vx\|_2 = 1$ and $\|\vx\|_1 \le 1$ can only be simultaneously satisfied by $1$-sparse vectors, namely the standard basis vectors $\pm \bm{e}_i$. The problem then reduces to finding the maximum diagonal entry of $\bm{A}$, which requires $O(n)$ time.

    If $R \le 1/\sqrt{n}$, then the maximum possible $L_1$ norm for any point on the sphere $\|\vx\|_2 = R$ is $\sqrt{n}R \le 1$. Consequently, the entire $L_2$ sphere lies strictly inside the $L_1$ ball, rendering the constraint $\|\vx\|_1 \le 1$ redundant. The problem degenerates to $\max_{\|\vx\|_2 = R} \vx^\top \bm{A} \vx$, which is equivalent to computing the largest eigenvalue of $\bm{A}$ and can be solved in polynomial time.
\end{proof}

While ($L_1$-$L_2$-QM) is tractable for all $R\not\in(1/\sqrt{n}, 1)$, we show that within the active geometric intersection $1/\sqrt{n}<R<1$, it is fundamentally NP-hard.

\begin{theorem}
\label{thm:l1l2_nphard}
The problem ($L_1$-$L_2$-QM) is NP-hard.
\end{theorem}

\begin{proof}
We reduce from the $k$-clique problem. Given an undirected graph $G=(V,E)$, deciding whether $G$ contains a clique of size $k$ is NP-complete \cite{karp2009reducibility}. Let $\bm{A}$ be the adjacency matrix of $G$ and $k$ be an integer with $1<k<n$. Set $R = 1/\sqrt{k}$, which satisfies $1/\sqrt{n} < R < 1$.

Since $\bm{A}$ is non-negative, we may restrict to $x_i\ge0$ without loss of generality, giving the equivalent formulation
\[
(L_1\text{-}L_2\text{-QM}^+)\ \max_{\vx\in\mathbb{R}^n}\ \vx^\top\bm{A}\vx\quad\text{s.t.}\quad \|\vx\|_1\le1,\ x_i\ge0 ~\forall i,\ \|\vx\|_2=R.
\]
For any feasible $\vx$, expanding $\|\vx\|_1^2$ gives
\[
\|\vx\|_1^2 = \Bigl(\sum_{i=1}^n x_i\Bigr)^2 = \sum_{i=1}^n x_i^2 + \sum_{i\neq j} x_i x_j = \|\vx\|_2^2 + \sum_{i\neq j} x_i x_j.
\]
Since $\bm{A}$ is a loop-less adjacency matrix,
\[
\vx^\top\bm{A}\vx = \sum_{i\neq j}A_{ij} x_i x_j \le \sum_{i\neq j} x_i x_j.
\]
Combining the two relations, for any feasible solution of ($L_1$-$L_2$-QM$^+$), we have
\[
\vx^\top\bm{A}\vx \le \|\vx\|_1^2 - \|\vx\|_2^2 \le 1 - R^2.
\]
Let $v(\cdot)$ denote the optimal value of a problem.
We will prove that $v(L_1\text{-}L_2\text{-QM}^+)= 1-R^2$ if and only if $G$ has a clique of size $k$.

\textbf{Sufficiency:}
Suppose $G$ contains a clique $C$ of size $k$. Consider the uniform vector $\vx^*$ defined by $x_i^* = 1/k$ for $i \in C$ and $x_i^* = 0$ otherwise. Then $\|\vx^*\|_1 = 1$ and $\|\vx^*\|_2 = 1/\sqrt{k} = R$, so $\vx^*$ is feasible for ($L_1$-$L_2$-QM$^+$). Consequently,
\[
{\vx^*}^\top \bm{A} \vx^* = \sum_{i,j \in C,~ i \neq j} \frac{1}{k^2} = \frac{k(k-1)}{k^2} = 1 - \frac{1}{k} = 1 - R^2.
\]
Thus, the upper bound $1 - R^2$ is attained and hence $v(L_1\text{-}L_2\text{-QM}^+) = 1 - R^2$.

\textbf{Necessity:}
Suppose, to the contrary, that $G$ contains no $k$-clique, i.e., $\omega(G) \le k-1$. We will show that $v(L_1\text{-}L_2\text{-QM}^+) < 1 - R^2$.

For any feasible $\vx$ of ($L_1$-$L_2$-QM$^+$), we have $\vx / \|\vx\|_1 \in \Delta$, the standard simplex. By the classical Motzkin–Straus theorem (Lemma \ref{theorem:Motzkin-Straus}),
\[
\left(\frac{\vx}{\|\vx\|_1}\right)^\top \bm{A} \left(\frac{\vx}{\|\vx\|_1}\right) \le 1 - \frac{1}{\omega(G)}.
\]
Multiplying both sides by $\|\vx\|_1^2$ and using $\|\vx\|_1 \le 1$, we obtain
\[
\vx^\top \bm{A} \vx \le \left(1 - \frac{1}{\omega(G)}\right) \|\vx\|_1^2 \le 1 - \frac{1}{\omega(G)} \le 1 - \frac{1}{k-1} < 1 - R^2.
\]
Hence, $v(L_1\text{-}L_2\text{-QM}^+) < 1 - R^2$.
 
We have shown that, given a graph $G$ and an integer $k$ with $1<k<n$, one can choose $R = 1/\sqrt{k}$  and construct the corresponding instance of ($L_1$-$L_2$-QM$^+$). Solving this instance yields its optimal value $v$. If $v = 1 - R^2$, then $G$ contains a $k$-clique; otherwise, $G$ contains no $k$-clique. Since the $k$-clique problem is NP-complete, ($L_1$-$L_2$-QM$^+$) and consequently ($L_1$-$L_2$-QM) are NP-hard.
\end{proof}

Similar to Lemma \ref{lem:range_R}, one can show that the problem ($L_1$-$L_2$-QM-Ineq) is solvable in polynomial time when $R \le 1/\sqrt{n}$. For $R \ge 1$, it reduces to (QPL1), which is NP-hard \cite{hsia2014complexity}. The remaining regime $1/\sqrt{n} < R < 1$ is the focus of this paper.

\begin{corollary}
\label{cor:inequality}
For $1/\sqrt{n} < R < 1$,  the problem ($L_1$-$L_2$-QM-Ineq) is NP-hard.
\end{corollary}

\begin{proof}
We use the same reduction as in Theorem \ref{thm:l1l2_nphard}. Given a graph $G$ and an integer $k$, set $R = 1/\sqrt{k}$ and construct the instance with adjacency matrix $\bm{A}$.

If $G$ contains a $k$-clique, the uniform vector $\vx^*$ (with $x_i^* = 1/k$ on the clique) satisfies $\|\vx^*\|_1 = 1$ and $\|\vx^*\|_2 = R$, hence is feasible for ($L_1$-$L_2$-QM-Ineq), achieving objective value $1-R^2$. Thus $v(L_1\text{-}L_2\text{-QM-Ineq}) \ge 1-R^2$.

If $G$ contains no $k$-clique, the same argument as in Theorem \ref{thm:l1l2_nphard} shows that every feasible solution satisfies $\vx^\top \bm{A} \vx < 1 - R^2$, so $v(L_1\text{-}L_2\text{-QM-Ineq}) < 1-R^2$.

Therefore, $v(L_1\text{-}L_2\text{-QM-Ineq}) \ge 1-R^2$ if and only if $G$ contains a $k$-clique. Solving ($L_1$-$L_2$-QM-Ineq) would decide the $k$-clique problem, and the problem is NP-hard.
\end{proof}

\section*{Funding}
This work was supported by  the National Natural Science Foundation of China (No.
12171021).
\section*{Data Availability}
No datasets were generated or analysed during the current study.
\section*{Declarations}
Ethical Approval Not applicable.
Competing interests The authors declare no competing interests.

\bibliographystyle{unsrt}
\bibliography{refs}

\end{document}